\documentclass[12pt]{article}
\usepackage{tikz}
\usetikzlibrary{calc,through,backgrounds}
\usepackage{graphicx,enumerate}
\usepackage{subfig} 
\usepackage[justification = centering, labelsep =period]{caption} 

\usepackage{multicol}
\usepackage{ytableau}
\usepackage{bm,amssymb,amsthm,amsfonts,amsmath,latexsym,cite,color, psfrag,graphicx,ifpdf,pgfkeys,pgfopts,xcolor,extarrows}

\newtheorem{theo}{Theorem}[section]

\newtheorem{lem} [theo]{Lemma}
\newtheorem{coro}[theo]{Corollary}

\newtheorem{re}[theo]{Remark}
\newtheorem{conj}[theo]{Conjecture}

\allowdisplaybreaks

\makeatletter \@addtoreset{equation}{section}
\@addtoreset{theo}{}\makeatother

\usepackage{hyperref}

\def\r2{r_2(i_b,i_c,j,j')}

\begin{document}
\begin{center}
{\Large\bf Richardson tableaux and noncrossing partial matchings}

\vskip 6mm
{\small   }
Peter L. Guo 

\end{center}

\begin{abstract}
Richardson tableaux   are a remarkable subfamily of standard Young tableaux introduced by Karp and Precup in order to index the irreducible components of  Springer fibers   equal to Richardson varieties. We show that the set of insertion tableaux of noncrossing partial matchings on $\{1,2,,\ldots, n\}$ by applying the Robinson--Schensted  algorithm coincides with the set of Richardson tableaux of size $n$. 
This leads to a natural one-to-one correspondence  between  the set of Richardson tableaux of size $n$ and the set of Motzkin paths with $n$ steps, in response to a problem proposed by Karp and Precup. As consequences, we recover some known and establish  new   properties   for Richardson tableaux.   Especially, we relate the $q$-counting of Richardson tableaux to $q$-Catalan numbers.

\end{abstract}

\vskip 3mm




\newcommand*\cir[1]{\tikz[baseline=(char.base)]{
            \node[shape=circle,draw,inner sep=2pt] (char) {#1};}}

\section{Introduction}

Let $\lambda=(\lambda_1\geq \lambda_2\geq \cdots\geq \lambda_\ell>0)$ be a partition of $n$ with $\ell$ parts, and $N$ be an  $n\times n$ nilpotent matrix  which is in Jordan canonical form with block
 sizes $\lambda_1,\ldots, \lambda_\ell$ weakly decreasing. 
The Springer fiber $\mathcal{B}_\lambda$ indexed by $\lambda$  is  a subvariety of the flag variety $\mathrm{Fl}_n(\mathbb{C})=\{F_0\subset F_1\subset \cdots \subset F_n=\mathbb{C}^n\colon \mathrm{dim}(F_j)=j\}$   \cite{Springer}.  Precisely,   
\[
\mathcal{B}_\lambda=\{F_0\subset F_1\subset \cdots \subset F_n=\mathbb{C}^n\in \mathrm{Fl}_n(\mathbb{C})\colon N(F_j)\subset F_{j-1}\ \text{for $1\leq j\leq n$}\}.
\]
The cohomology  ring  of a Springer fiber  carries a representation  of the symmetric group  where the top-dimensional cohomology is irreducible \cite{Springer-3} .
By the work of Spaltenstein \cite{Spa}, the irreducible components of $\mathcal{B}_\lambda$ are indexed by standard Young tableaux $T$ of shape $\lambda$. Let $\mathcal{B}_T$ denote the associated   irreducible component  of $\mathcal{B}_\lambda$. It remains an open problem, as posed by Springer \cite{Springer-2}, about how to expand  the cohomology class $[\mathcal{B}_T]$ in terms of the basis of  Schubert classes in the cohomology  of $\mathrm{Fl}_n(\mathbb{C})$. 

Recently, motivated the study of totally nonnegative Springer fibers by Lusztig  \cite{Lusztig}, Karp and Precup \cite{KP}  found a complete criterion of when $\mathcal{B}_T$ is equal to a Richardson variety $R_{v, w}$, namely, the intersection of a Schubert variety $X_w$ and an opposite Schubert variety $X^v$.
To be specific, they introduced a new subfamily of standard   tableaux, called Richardson tableaux, and showed that $\mathcal{B}_T$ is  a Richardson variety $R_{v_T,w_T}$ (where $v_T$ and $w_T$ are permutations determined by $T$) if and only if $T$ is a Richardson tableau \cite{KP}. This family  includes as special cases previous work on  Springer fibers studied for example in 
\cite{BZ-o,Gue,GZ}.
When $T$ is a Richardson tableau, the Schubert class expansion of  $[\mathcal{B}_T]$   is equivalent to  computing the  Schubert structure constants:
\[
[R_{v_T, w_T}]=\sum_{u} c^{w_T}_{v_T,u} [X^u].
\]
Spink and  Tewari \cite{ST} later  discovered that such   structure constants can be   combinatorially computed by   iterating the Pieri rule  due to Sottile \cite{Sottile}. This gives an answer to  Springer's problem in the case when $T$ is a Richardson tableau.

Various properties concerning Richardson tableaux have been developed by Karp and Precup \cite{KP}. Their enumeration is surprisingly elegant:  the number of Richardson tableaux of size $n$ is equal to the number of Motzkin paths with $n$ steps. Their proof is algebraic, relying on the generating function of Richardson tableaux. Karp and Precup \cite[Problem 4.5]{KP} asked for an explicit correspondence  between Richardson tableaux and  Motzkin paths. In this note, we present  such a construction which is achieved  by applying the Robinson--Schensted (RS) algorithm to noncrossing partial matchings. 
From the RS algorithm's point of view, some properties of Richardson tableaux are fairly natural. Using this correspondence, we relate a $q$-counting of Richardson tableaux to $q$-Catalan numbers.   We hope that the RS algorithm could shed some new   light on the  study of the geometry of Springer fibers.

We proceed to introduce some notions needed to describe our results. 
A partial matching on $[n]$ is a set partition of $[n]$ such that each block contains either a single number or exactly two numbers. Notice that partial matchings on $[n]$ can be identified  with involutions among the permutations of $[n]$. That is, for a partial matching on $[n]$, the associated  involution $w$ is defined by setting $w(i)=i$ if $\{i\}$ is a singleton, and $w(i)=j$ if $\{i,j\}$ is a two-element block. For example, the involution corresponding to the partial matching  $\{\{1,5\}, \{2\}, \{3,4\}, \{6\}, \{7,8\}\}$ is $52431687$.

We may represent a partial matching on $[n]$ by its diagram. Specifically, list  $n$ dots labeled   $1,2,\ldots,n$ from left to right, and draw an arc between $i$ and $j$  if ${i,j}$ lie in the same block. 
For example, Figure \ref{fig-1} illustrates  the diagrams of the partial matchings $\{\{1,5\}, \{2\}, \{3,4\}, \{6\}, \{7,8\}\}$ and $\{\{1,5\}, \{2\}, \{3,6\}, \{4,8\}, \{7\}\}$. 
\begin{figure}[h t]
  \begin{center}
    \begin{tikzpicture}[>=stealth,scale=0.6]
    \foreach \i in {1,...,8} {
        \node at (\i,0) {\i};
        \fill (\i,0.4) circle (1.2pt); 
    }

    \draw[-,bend left=40] (1,0.4) to (5,0.4);
    \draw[-,bend left=30] (3,0.4) to (4,0.4);
    \draw[-,bend left=30] (7,0.4) to (8,0.4);

\end{tikzpicture}
\qquad
\begin{tikzpicture}[>=stealth,scale=0.6]
    \foreach \i in {1,...,8} {
        \node at (\i,0) {\i};
        \fill (\i,0.4) circle (1.2pt); 
    }

    \draw[-,bend left=40] (1,0.4) to (5,0.4);
    \draw[-,bend left=30] (3,0.4) to (6,0.4);
    \draw[-,bend left=30] (4,0.4) to (8,0.4);

\end{tikzpicture}
\end{center}
\caption{Two partial matchings on $[8]$.}\label{fig-1}
\end{figure}
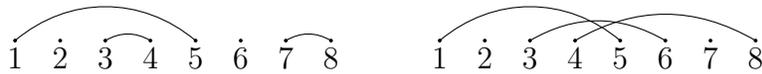
We  use $(i,j)$, where $i<j$, to denote an arc connecting $i$ and $j$. 
Two arcs $(i_1,j_1)$ and $(i_2, j_2)$ form a crossing if $i_1<i_2<j_1<j_2$, and form a nesting if $i_1<i_2<j_2<j_1$. Crossings and nestings were extensively studied in the work  of Chen, Deng, Du, Stanley and Yan \cite{Chen-1}. A partial matching is called noncrossing if it has no crossing. For example, the left picture in Figure \ref{fig-1} is  noncrossing. Noncrossing partial matchings have been  employed  as a basic combinatorial model to study the RNA secondary structures \cite{Chen}. Noncrossing matchings have also appeared in total nonnegativity, see for example   the stratification of the   space of cactus networks \cite{Lam}.

As a classical result, the RS algorithm associates with a permutation $w$ of $[n]$   a pair  $(T,T')$ of standard Young tableaux of size $n$ with the same shape, where $T$ (resp., $T'$) is  called the insertion (resp., recording)  tableau  of $w$. Denote $w\xrightarrow{\mathrm{RS}}(T,T')$.   Sch\"utzenberger's theorem \cite{Schu-1} says that  the inverse $w^{-1}$ of $w$ satisfies $w^{-1}\xrightarrow{\mathrm{RS}}(T',T)$.   This symmetry  implies that $w$ is an involution if and only if $T=T'$. Consequently,   the RS algorithm sets up a bijection between involutions of $[n]$ are   standard  tableaux of size $n$. 

In the following, we shall not distinguish between a partial matching on $[n]$ and an involution of $[n]$. For instance,  the insertion tableau of a partial matching refers to the insertion tableau of the corresponding involution, a noncrossing involution means an involution to which the corresponding partial matching is noncrossing. 

Our main result is stated as follows. 

\begin{theo}\label{thm-1}
    The set of insertion tableaux of noncrossing partial matchings on $[n]$ is exactly the set of Richardson tableaux of size $n$ (the definition of a Richardson tableau will be given in Section \ref{sect-2}).  
\end{theo}

We give     applications of Theorem \ref{thm-1}. The first one  explains that Theorem \ref{thm-1} gives an answer to  \cite[Problem 4.5]{KP} raised  by Karp and Precup.

\begin{re}\label{re-0jy}
There is an obvious bijection between noncrossing partial matchings on $[n]$ and Motzkin paths with $n$ steps. Recall that a Motzkin path with $n$ steps is a path  from $(0,0)$ to $(n,0)$, using possibly   up steps $(1,1)$, down steps $(1,-1)$ and horizontal steps $(1,0)$, which never goes strictly below   the $x$-axis.  Given a noncrossing partial  matching   on $[n]$, we construct  a Motzkin path   as follows. For $1\leq i\leq n$, the  $i$-th step of the path  is $(1,1)$ (resp., $(1,-1)$) if $i$ is a left  (resp., right) endpoint  of an arc, and is $(1,0)$ if $\{i\}$ is a singleton.  So Theorem \ref{thm-1} yields an explicit  bijection between  Motzkin paths and  Richardson tableaux, answering \cite[Problem 4.5]{KP}.
\end{re}

As the  main result in Section 3 of  \cite{KP}, it was shown that the set of Richardson tableaux is invariant  under the evacuation map.    This phenomenon will be clear from Theorem \ref{thm-1}. 
Let $\mathrm{evac}(T)$ denote  the evacuation  tableau of $T$, see Section \ref{sect-2} for the definition. 

\begin{coro}[{\cite[Corollary 3.23]{KP}}]\label{coro-2}
The set of Richardson tableaux is closed under the evacuation map. 
That is, if $T$ is a Richardson tableau, then so is $\mathrm{evac}(T)$. 

\end{coro}

\begin{proof}
Let $w_0=n\cdots 21$ be the longest permutation of $[n]$.  It is  known that if $w\xrightarrow{\mathrm{RS}}(T,T')$, then $w_0ww_0\xrightarrow{\mathrm{RS}}(\mathrm{evac}(T),\mathrm{evac}(T'))$ \cite[Theorem A1.2.10]{Stanley}.
Assume now that  $w$ is a noncrossing involution. 
Note that $w_0ww_0$ is   also an involution. Moreover, the diagram of $w_0ww_0$ is obtained from that of $w$ after  a reflection along a vertical line passing  the midpoint, see Figure \ref{fig-1x} for an illustration.
\begin{figure}[h t]
  \begin{center}
    \begin{tikzpicture}[>=stealth,scale=0.6]
    \foreach \i in {1,...,8} {
        \node at (\i,0) {\i};
        \fill (\i,0.4) circle (1.2pt); 
    }

    \draw[-,bend left=40] (1,0.4) to (5,0.4);
    \draw[-,bend left=30] (3,0.4) to (4,0.4);
    \draw[-,bend left=30] (7,0.4) to (8,0.4);
    \draw[very thick] (4.5,-1) -- (4.5,1.5);
\end{tikzpicture}
\qquad
\begin{tikzpicture}[>=stealth,scale=0.6]
    \foreach \i in {1,...,8} {
        \node at (\i,0) {\i};
        \fill (\i,0.4) circle (1.2pt); 
    }

    \draw[-,bend left=40] (1,0.4) to (2,0.4);
    \draw[-,bend left=30] (4,0.4) to (8,0.4);
    \draw[-,bend left=30] (5,0.4) to (6,0.4);
    \draw[very thick] (4.5,-1) -- (4.5,1.5);

\end{tikzpicture}
\end{center}
\caption{Diagrams of $w$ and $w_0ww_0$ for $w=52431687$.}\label{fig-1x}
\end{figure}
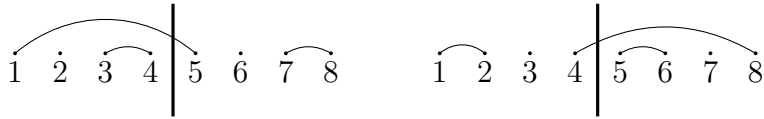
This implies  $w_0ww_0$ is also   noncrossing. By Theorem \ref{thm-1}, $\mathrm{evac}(T)$, as the insertion tableau of $w_0ww_0$, is a Richardson tableau. 
\end{proof}

A prime Richardson tableau is a Richardson tableau which cannot be decomposed into a concatenation of smaller Richardson tableaux, see Section \ref{sect-3} for detailed information. For $n\geq 2$, the shape of a prime Richardson tableau must be of the form    $\lambda=(\lambda_1,\ldots, \lambda_{\ell-2},1,1)$, that is, the last two rows have length equal to one \cite[Corollary 3.15]{KP}.  Theorem \ref{thm-1} leads to an alternative explanation of  the following  relation.

\begin{coro}[{\cite[Corollary 4.3]{KP}}]\label{coro-22}
The number of prime Richardson tableaux of shape $\lambda=(\lambda_1,\ldots, \lambda_{\ell-2},1,1)$ equals the number of
 Richardson tableaux of shape $(\lambda_1,\ldots, \lambda_{\ell-2})$. 
\end{coro}  

\begin{proof}
    See Remark \ref{re-xj}.
\end{proof}

We say that a column of a standard tableau is odd (resp., even) if it has an odd (resp., even) number of boxes. Let $C_n=\frac{1}{n+1}{2n \choose n}$ be the $n$-th Catalan number. 

\begin{coro}\label{coro-oy}
Let $0\leq k\leq n$ with $k \equiv n \pmod{2}$.    The number of Richardson tableaux of size $n$ and with $k$ odd columns   is equal to 
\begin{align}\label{bg-o0}
    {n\choose k} C_{\frac{n-k}{2}}.
\end{align}
\end{coro}

\begin{proof}
A theorem due to Sch\"utzenberger \cite{Schu}   states that the  number of fixed points of an involution  equals the number of odd columns in its insertion tableau, see also \cite{Bei}. So it is equivalent to showing that the  number of noncrossing partial matchings on $[n]$ with $k$ singletons  is ${n\choose k} C_{\frac{n-k}{2}}.$ To see this, there are ${n\choose k}$ ways to choose $k$ elements of $[n]$ as singletons, and    the remaining $n-k$ elements form  noncrossing perfect matchings which are counted by the Catalan number $C_{\frac{n-k}{2}}$.
\end{proof}

We also consider   $q$-anologues of Corollary \ref{coro-oy}, which turn out to be  closely related to $q$-Catalan numbers. It is worth pointing out that $q$-Catalan numbers (more generally, rational $q$-Catalan numbers) have emerged in counting points of certain Richardson varieties   over a finite field  by Galashin and Lam \cite{GL}. 
For a permutation $w$ of $[n]$, a position  $i\in [n-1]$ is  a descent (resp.,  ascent) if $w(i)>w(i+1)$ (resp., $w(i)<w(i+1)$). Let $\mathrm{maj}(w)$ (resp., $\mathrm{comaj}(w)$) be the sum of descents  (resp., ascents) of $w$.   For a standard  tableau $T$ of size $n$, an entry $i\in [n-1]$ is a descent of $T$ if $i+1$  appears in a row strictly below $i$, and otherwise, we call $i$  an ascent. Similarly,  we use $\mathrm{maj}(T)$  (resp., $\mathrm{comaj}(T)$) to denote   the  sum of   descents (resp., ascents) of $T$. Clearly,  \[
\mathrm{maj}(w)+\mathrm{comaj}(w)=\mathrm{maj}(T)+\mathrm{comaj}(T)={n\choose 2}.
\]

A Richardson tableau is called  even if it has no odd columns. 
By Corollary \ref{coro-oy}, an even  Richardson tableau must have even size. 
For $m\geq 0$, let $[m]_q=\frac{1-q^m}{1-q}=1+q+\cdots+q^{m-1}$, and $[m]_q!=[1]_q[2]_q\cdots [m]_q$. We establish  the following $q$-counting for even Richardson tableaux.

\begin{coro}
For $n\geq 1$,  write  $\mathrm{ERT}(2n)$ for  the set of even Richardson tableaux of size $2n$. Then we have
\begin{align}\label{qcata}
\sum_{T\in \mathrm{ERT}(2n)} q^{\mathrm{comaj}(T)}= \frac{1}{[n+1]_q}  {2n \brack n}_q=\frac{[2n]_q!}{[n+1]_q!\cdot [n]_q!}, 
\end{align}
which is the $n$-th $q$-Catalan number. For $0\leq m\leq n-1$, let $\mathrm{ERT}(2n,m)$ denote the set of even Richardson tableaux of size $2n$ with $m$ ascents. Then we have the following refinement via the $q$-Narayana number:  
\begin{align}\label{qnara}
\sum_{T\in \mathrm{ERT}(2n,m)} q^{\mathrm{comaj}(T)}= \frac{1}{[n]_q}   {n \brack m }_q  {n \brack m+1 }_q q^{m^2+m}. 
\end{align}
\end{coro}

\begin{proof}
We first prove    \eqref{qcata}. Let $T$ be an even Richardson tableau of size $2n$, which, by the proof of Corollary \ref{coro-oy},  corresponds to   a noncrossing involution $w$ of $[2n]$ without fixed points.  By Remark \ref{re-0jy}, $w$ corresponds to a Dyck path  $D=D_1D_2\cdots D_{2n}$ where  $D_i$ is either an up step $(1,1)$ or a down step $(1,-1)$. We say $i\in [2n-1]$ is a valley point if $D_i=(1,-1)$ and $D_{i+1}=(1,1)$. Let $\mathrm{maj}(D)$ denote the sum of valley points of $D$. A famous   $q$-counting  due to MacMahon \cite{Mac} (see also \cite{FH}) says that
\begin{align}\label{qcata1}
    \sum_{D} q^{\mathrm{maj}(D)}=\frac{1}{[n+1]_q}  {2n \brack n }_q,
\end{align}
where the sum runs over Dyck paths with $2n$ steps. 

By \cite[Lemma 7.23.1]{Stanley},  it is known that $i$ is a deccent of $w$ if and only if it is a descent of $T$, and so we have   $\mathrm{maj}(w)=\mathrm{maj}(T)$, or equivalently, $\mathrm{comaj}(w)=\mathrm{comaj}(T)$. 
Moreover, we have the  following crucial observation: 
\[
\text{$i$ is a valley point of $D$ if and only if it is an ascent of $w$.}
\]
Therefore,  we have  $\mathrm{maj}(D)=\mathrm{comaj}(w)$, and hence 
 $\mathrm{maj}(D)=\mathrm{comaj}(T)$. This, along with \eqref{qcata1}, leads to \eqref{qcata}. 

To conclude the identity in \eqref{qnara}, it is enough to notice  that 
 \begin{align}\label{qna} 
    \sum_{D} q^{\mathrm{maj}(D)}=\frac{1}{[n]_q}   {n \brack m }_q  {n \brack m+1 }_q q^{m^2+m}, 
\end{align}
where $D$ is restricted to Dyck paths with $2n$ steps and $m$ valley points. The formula in \eqref{qna} was stated by MacMahon without proof, and see \cite[Section 4]{FH} for a proof. 
\end{proof}

We conjecture that Corollary \ref{coro-oy} has the following $q$-analogue which has been checked  for $n$ up to $14$. 

\begin{conj}
For $0\leq k\leq n$ with $k \equiv n \pmod{2}$,  let $\mathrm{RT}(n,k)$ denote the set of   Richardson tableaux of size $n$ and with $k$ odd columns.
Then 
\begin{align}\label{conj-o}
 \sum_{T\in \mathrm{RT}(n,k)}q^{\mathrm{comaj}(T)}=q^{k\choose 2} \frac{}{} {n \brack k}_q C_{\frac{n-k}{2}}(q), 
\end{align}
where $C_m(q)=\frac{1}{[m+1]_q}  {2m \brack m }_q$ denotes the $m$-th $q$-Catalan number. Equivalently,
\begin{align*} 
 \sum_{w}q^{\mathrm{comaj}(w)}=q^{k\choose 2} \frac{}{} {n \brack k}_q C_{\frac{n-k}{2}}(q), 
\end{align*}
where the sum is over all noncrossing involutions of $[n]$ with exactly $k$ fixed points. 
\end{conj}

Setting $k=0$ and replacing $n$ by $2n$, the formula in  \eqref{conj-o} specializes to 
 \eqref{qcata}. 

\begin{re} A formula for   $q$-counting  the statistic $\mathrm{maj}(T)$ of Richardson tableaux of any fixed shape has been explored  in \cite[Theorem 4.8]{KP}. 
We do not know how to deduce   this formula  by using noncrossing involutions. 
\end{re}

By Corollary \ref{coro-2}, the evacuation map  can be regarded  as a permutation of  Richardson tableaux of size $n$. In fact, this is an involution since $\mathrm{evac}(\mathrm{evac}(T))=T$. 
Let $M_n$   denote the $n$-th Motzkin number, counting   Motzkin paths with $n$ steps. Here we set $M_0=1$.

\begin{coro}
The number of Richardson tableaux of size $n$ which are fixed by the evacuation map is equal to 
\begin{equation}\label{chj-i}
M_{\left\lfloor \frac{n}{2} \right\rfloor}+\sum_{k\geq 1} \sum_{1\leq i_1<\cdots<i_k\leq \left\lfloor \frac{n}{2} \right\rfloor}\prod_{p=1}^{k+1} M_{i_p-i_{p-1}-1},
\end{equation}
where we set $i_0=0$ and $i_{k+1}=\left\lfloor \frac{n}{2} \right\rfloor+1$.
\end{coro}

\begin{proof}
In view of  the proof of Corollary \ref{coro-2}, we need to count   noncrossing partial matchings $P$ on $[n]$ which are invariant after a reflection along the vertical line, denoted $L$, passing the midpoint. We first consider the case when $n$ is even. Each arc $(i,j)$ of $P$ either lies on the left-hand/right-hand side of $L$, or intersects with $L$.
In the latter case, we have $i\leq \frac{n}{2}<j$. Since $P$ is noncrossing and symmetric with respect to $L$, we must have $j=n+1-i$, that is, the arc $(i,j)$ must be invariant after a reflection along $L$. Such an arc $(i,j)$, where $1\leq i\leq \frac{n}{2}$ and $j=n+1-i$, is called a big arc.

Note that  the configuration of arcs on the left-hand side of $L$ completely determines $P$. 
Suppose that $P$ has $k$ big arcs. 
If $k=0$, then  the dots $1,2,\ldots, \frac{n}{2}$   form  a noncrossing partial mathching. Such noncrossing partial matchings are counted by $M_{\frac{n}{2}}$. We next consider $k\geq 1$. Let $1\leq i_1<i_2<\cdots<i_k\leq \frac{n}{2}$   be the left endpoints of these big arcs. Notice that for $1\leq p\leq k+1$, the dots $i_{p-1}+1,\ldots, i_p-1$ between $i_{p-1}$ and $i_p$ form a noncrossing partical matching. Here we set $i_0=0$ and $i_{k+1}=\left\lfloor \frac{n}{2} \right\rfloor+1$. So these   noncrossing partial matchings contribute 
\[
\prod_{1\leq p\leq k+1} M_{i_p-i_{p-1}-1}.
\]
Combining the above verifies the formula in \eqref{chj-i}.

In the case when $n$ is odd, we notice that the midpoint $\frac{n+1}{2}$ must be a singleton. The remaining  arguments are completely  analogous
to the even case. 
\end{proof}

This paper  is organized as follows. In Section \ref{sect-2}, we review the RSK algorithm and the evacuation map operating on tableaux. The definition of Richardson tableaux is also included. In Section \ref{sect-3}, we present a proof of Theorem \ref{thm-1}. 

\subsection*{Acknowledgements}
I am grateful to Steven Karp for very helpful  comments on an earlier version. 
This work was  supported by the National Natural Science Foundation of China (No. 12371329) and the Fundamental Research Funds for the Central Universities (No. 63243072).

\section{RS algorithm and evacuation map}\label{sect-2}

We give an overview of the Robinson--Schensted  algorithm, the evacuation map  as well as the definition of Richardson tableaux.

 We mainly follow the terminology in  \cite[Chapter 7]{Stanley}.  Let $\lambda=(\lambda_1\geq   \cdots\geq \lambda_\ell>0)$ be a partition of $n$ with $\ell$ parts. One can  identify $\lambda$ with its Young diagram, namely, a left-justified array of boxes with $\lambda_i$ boxes in row $i$. A semistandard Young tableau $T$ of shape $\lambda$ is a filling of positive integers into the boxes of $\lambda$ such that (1) each box  receives exactly one number, (2) the numbers in each row are weakly increasing from left to right, and (3) the numbers along each column are strictly increasing from top to bottom.  A semistandard tableau $T$ is called standard if the integers appearing in $T$ are exactly $1,2,\ldots, n$. Figures \ref{fig-2} illustrates two  semistandard Young tableaux of shape $(3,2,2)$, where the right one  is a standard tableau. 
\begin{figure}[h t]
\begin{center}
    \begin{ytableau}
1 & 1 & 2 \\
2 & 3 \\
3 & 4
\end{ytableau}
\qquad
\begin{ytableau}
1 & 2 & 4\\
3 & 5 \\
6 & 7
\end{ytableau}
\end{center}
\caption{Two semistandard Young tableaux of size $7$.}\label{fig-2}
\end{figure}

Given a semistandard tableau $T$ and positive integer $i$, the Robinson--Schensted-Knuth (RSK) algorithm produces  a new semistandard tableau, denoted $T\leftarrow i$, by inserting $i$ into $T$. The procedure relies on a row bumping operation. First, look at the first row $R$ of $T$. If $i$ is greater than or equal to every entry in $R$, then place $i$ at the end of $R$ and stop. Otherwise, locate the leftmost entry, say $i'$, in $R$ which is greater than $i$, and replace this entry by $i$. In this case, we continue to insert the bumped entry $i'$ into the second row by applying the same bumping operation. Eventually, the procedure will stop, and the resulting tableau is denoted  $T\leftarrow i$. In Figure \ref{fig-k}, we illustrate the procedure of inserting $i=1$ into the left semitandard tableau in Figure \ref{fig-2}. 
\begin{figure}[h t]
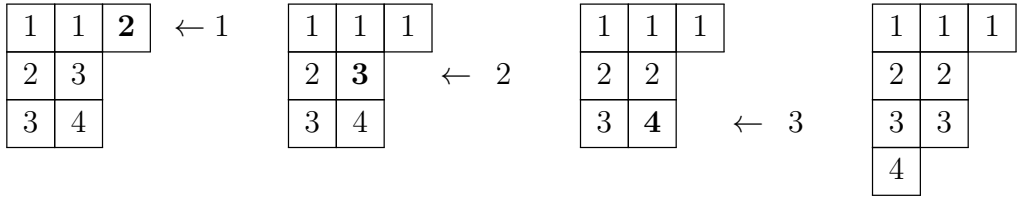

\begin{center}
    \begin{ytableau}
1 & 1 & {\bf 2}  & \none[~~~\leftarrow] &\none[1]\\
2 & 3 \\
3 & 4
\end{ytableau}~~~
    \begin{ytableau}
1 & 1 & 1 \\
2 & {\bf 3} & \none[] & \none[\leftarrow] &\none[2]\\
3 & 4
\end{ytableau}
~~~
    \begin{ytableau}
1 & 1 & 1 \\
2 & 2  \\
3 & {\bf 4} &\none[] & \none[\leftarrow] &\none[3]
\end{ytableau}
 ~~~
    \begin{ytableau}
1 & 1 & 1 \\
2 & 2  \\
3 & 3  \\
4
\end{ytableau}
\end{center}
\caption{An illustration of the RSK algorithm. }\label{fig-k}
\end{figure}

Let $w=w(1)w(2)\cdots w(n)$ be a permutation  of $[n]$. The RSK algorithm can be used to generate a pair $(T, T')$ of standard tableaux of the same shape. It should be pointed out that in the case of permutations, the RSK algorithm is usually called Robinson--Schensted (RS) algorithm which is applied  to words without repeated entries.  Let $T_0$ be the empty tableau, and  $T_i=T_{i-1}\leftarrow w(i)$ for $1\leq i\leq n$. Set $T=T_n$. 
Let $\lambda^i$ be the shape of $T_i$. Notice that   $\lambda^i$ is obtained from   $\lambda^{i-1}$ by adding a box in some corner. Then $T'$ is the standard tableau such that the entry $i$ is assigned in the corner box of $\lambda^i$ not contained in $\lambda^{i-1}$.  
The standard tableaux $T$ and $T'$ are called the insertion and recording tableaux of $w$, respectively. Conversely, given a pair of standard tableaux of the same shape, one can uniquely  recover the   corresponding permutation. 

The above correspondence will be  denoted as $w\xrightarrow{\mathrm{RS}}(T,T')$.
A remarkable theorem due to Sch\"utzenberger states that if   $w^{-1}$ is the inverse of $w$, then  $w^{-1}\xrightarrow{\mathrm{RS}}(T',T)$ \cite[Section 7.13]{Stanley}. So, when $w$ is an involution (that is, $w^2$ is the identity), we have $T=T'$. This gives rise  to a bijection between the set of involutions   of $[n]$ and the set of standard tableaux of size $n$.

We proceed to describe the evacuation map. Let $T$ be standard  tableau of shape $\lambda$. The evacuation map acts on $T$ based on the jeu de taquin slide operations. The resulting tableau, denoted $\mathrm{evac}(T)$, is defined  as follows.  Remove the entry $1$ from  the upper left
 corner of $T$, and decrease the remaining entries by $1$. Next,  slide this empty box until it reaches an outer corner of $T$. The rule is described as follows. Among the boxes directly  below or directly  to the right of this empty box, locate the one filled with a smaller number, and then swap it with the empty box. Repeat the same procedure   until the empty box slides to a position which is a corner on the southeast border  of $\lambda$. Now the entries $1,\ldots, n-1$ form a standard tableau, denoted $\Delta(T)$, of shape $\mu$ which is a partition of size $n-1$. Put the number $n$ into the corner box of $\lambda$ which is not contained in $\mu$.
 Doing the same for $\Delta(T)$ allows us to put $n-1$ into a corner of $\mu$. Continuing this procedure, we get a standard tableau of shape $\lambda$ by filling  the entries $n,n-1,\ldots, 1$ accordingly. The resulting tableau is defined as $\mathrm{evac}(T)$. 

Recall that $w_0=n\cdots 2 1$ denotes the longest permutation. The conjugate permutation $w_0ww_0$ maps $i$ to $n+1-w(n+1-i)$ for $1\leq i\leq n$. Sch\"utzenberger \cite{Schu-1} proved the following connection, see also \cite[Theorem A1.2.10]{Stanley}.

\begin{theo}[{\cite{Schu-1}}]
If $w\xrightarrow{\mathrm{RS}}(T,T')$, then $w_0ww_0\xrightarrow{\mathrm{RS}}(\mathrm{evac}(T),\mathrm{evac}(T'))$.     
\end{theo}

From the  above theorem, it can be seen that the evacuation map is actually an involution on standard tableaux, meaning that $\mathrm{evac}(\mathrm{evac}(T))=T$. 

We end this section with the definition of Richardson tableaux. For a standard tableau $T$ of size $n$ and $1\leq j\leq n$, let $T_{\leq j}$ be the subtableau occupied by the entries $1,\ldots, j$. Let $r_j(T)$ be the row number in which $j$ appears in $T$. We say that $T$ is a Richardson tableau if it satisfies the following   condition 
\begin{itemize}
    \item[] {\bf Richardson condition}: For each $1\leq j\leq n$ with $r_j(T)\geq 2$, the largest entry of $T_{\leq j}$ in row $r_j(T)-1$ is greater than every entry of $T_{\leq j}$   lying  in or strictly below row  $r_j(T)$. 
\end{itemize}
There are several equivalent characterizations for Richardson tableaux   \cite[Theorem 1.5]{KP}. We collect some of them in the following theorem. 

\begin{theo}[{\cite[Theorem 1.5]{KP}}]
Let $T$ be standard tableau of shape $\lambda$. The following statements are equivalent.
\begin{itemize}
    \item[(1)] $T$ is a Richardson tableau.

    \item[(2)] The evacuation tableau $\mathrm{evac}(T)$ is a Richardson tableau.

    \item[(3)] For every entry $j$ in the second row of $T$, the
 entry $j-1$ appears in the first row, and the standard tableau obtained from  $T$ by deleting the first row (and standardizing  the
remaining  entries) is Richardson.

\item[(4)] Every evacuation slide route of $T$ is of shape `$L$'.

\item[(5)]  The irreducible compoment $\mathcal{B}_T$   of the Springer fiber $\mathcal{B}_\lambda$ is a Richardson variety.
\end{itemize}
    
\end{theo}

\section{Noncrossing involutions and Richardson tableaux}\label{sect-3}

In this section, we establish the relationship between the insertion tableaux of noncrossing involutions and Richardson tableaux, thus finishing  the proof of Theorem \ref{thm-1}. 

Let $u$ and $v$ be permutations of $[m]$ and $[n]$, respectively. The direct sum   $u\oplus v$ of $u$ and $v$ is  the permutation of $[m+n]$ defined by 
\[
u(1)\cdots u(m)\, v(1)+m\cdots v(n)+m,
\]
that is, concatenate  $u$ and $v$, and then increase each entry $v(i)$ by $m$. For example, for $u=2413$ and $v=21$, we have $u\oplus v=2413 65$. A permutation $w$ is called prime if it cannot be written as direct sum of smaller permutations. Clearly, each permutation can be uniquely written as a direct sum of  prime permutations. 

\begin{re}\label{re-3-1}
Let $w$ be an noncrossing  involution of $[n]$. Then $w$ is prime if and only if $w(1)=n$. Equivalently, in its diagram representation, there is an arc $(1,n)$ connecting $1$ and $n$, see Figure \ref{fig-x} for an illustration. Clearly, each noncrossing involution is a direct sum of prime noncrossing involutions. 
\end{re}  
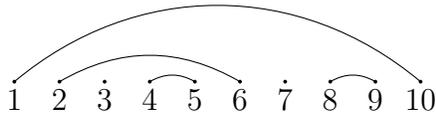
\begin{figure}[h t]
  \begin{center}
    \begin{tikzpicture}[>=stealth,scale=0.6]
    \foreach \i in {1,...,10} {
        \node at (\i,0) {\i};
        \fill (\i,0.4) circle (1.2pt); 
    }

    \draw[-,bend left=40] (1,0.4) to (10,0.4);
    \draw[-,bend left=30] (2,0.4) to (6,0.4);
    \draw[-,bend left=30] (4,0.4) to (5,0.4);
    \draw[-,bend left=30] (8,0.4) to (9,0.4);

\end{tikzpicture}
\end{center}
\caption{A prime noncrossing partial matching.}\label{fig-x}
\end{figure}

We now define the concatenation of  standard tableaux \cite[Section 3.3]{KP}. Let $T$ be a standard tableau of size $m$, and $T'$ be a standard tableau of size $n$. The concatenation $T\circ T'$ of $T$ and $T'$ is a standard tableau of size $m+n$ defined by increasing  each entry of $T'$ by $m$, and then concatenating the corresponding  rows of $T$ and $T'$. This is best understood by an example as illustrated in Figure \ref{fig-o}. 
\begin{figure}[h t]
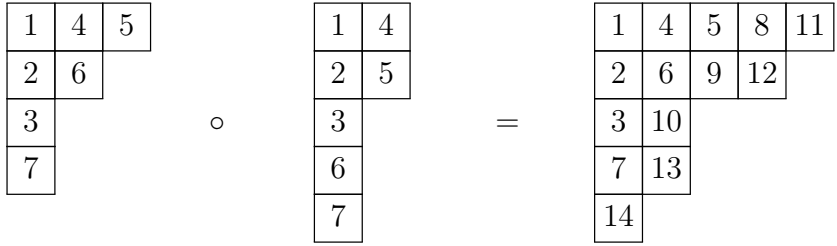

\begin{center}
    \begin{ytableau}
1 & 4 & 5 \\
2 & 6 \\
3 \\
7
\end{ytableau}
\quad\begin{ytableau}
    \none[]\\
    \none[]\\
    \none[\circ]
    
\end{ytableau} \qquad
\begin{ytableau}
1 & 4\\
2 & 5 \\
3\\
6\\
7
\end{ytableau}
\qquad\begin{ytableau}
    \none[]\\
    \none[]\\
    \none[=]
    \end{ytableau}\qquad
\begin{ytableau}
1 & 4 & 5 & 8 & 11\\
2 & 6 & 9 & 12 \\
3 & 10\\
7 & 13\\
14
\end{ytableau}
\end{center}
\caption{The concatenation of two standard tableaux.}\label{fig-o}
\end{figure}
Similarly, a standard tableau is prime if it cannot be decomposed into a concatenation of smaller  standard tableaux. Note that each standard tableau has a unique decomposition into prime pieces. 
For example, the left two standard tableaux in Figure \ref{fig-o} are both prime Richardson tableaux. 

\begin{re}
There is a geometric meaning of the concatenation \cite{Fre}:
\[\mathcal{B}_{T\circ T'}\simeq \mathcal{B}_{T}\times \mathcal{B}_{T'},\]
which could reduce the study of $\mathcal{B}_{T}$ to the   prime counterparts.  
\end{re}

By the construction of the RSK algorithm, the following fact 
is immediate.

\begin{lem}\label{lem-23}
Let $T$ and $T'$ be the insertion tableaux of $u$ and $v$, respectively. Then the insertion tableau of $u\oplus v$ is equal to $T\circ T'$. 
\end{lem}

Let $\mathrm{NC}(n)$ be the set of  noncrossing involutions of $[n]$, and $\mathrm{RT}(n)$ be the set of Richardson tableaux of size $n$. 
The following is restatement of Theorem \ref{thm-1}.

\begin{theo}\label{hg-x}
    For $n\geq 1$, there exists a bijection    $\Phi\colon \mathrm{NC}(n) \rightarrow\mathrm{RT}(n)$ by sending $w\in \mathrm{NC}(n)$ to its insertion tableau. 
\end{theo}

Before providing a proof of Theorem \ref{hg-x}, we list  the insertion tableaux of all partial matchings on $[4]$, as illustrated in Figure \ref{fig-xc}. 
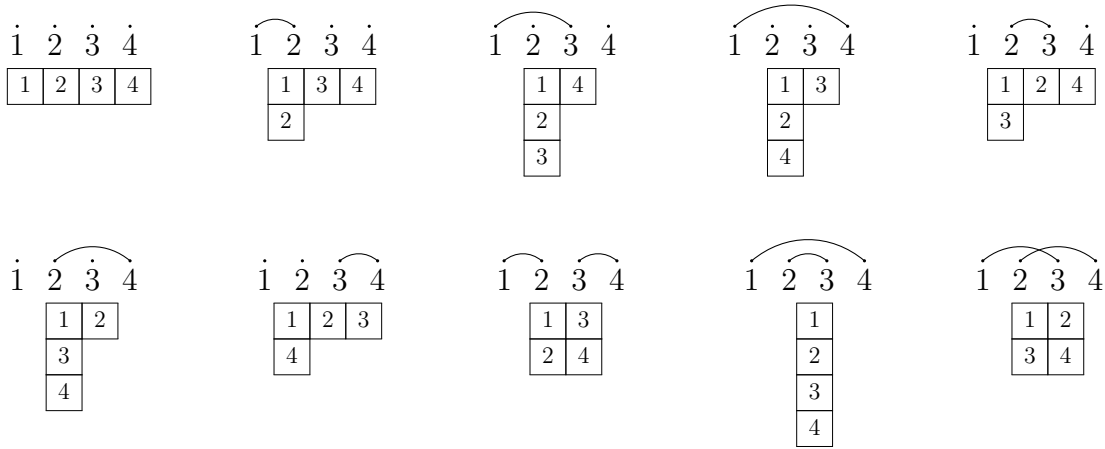
\begin{figure}[h t]
 \begin{tikzpicture}[>=stealth,scale=0.5]
    \foreach \i in {1,...,4} {
        \node at (\i,0) {\i};
        \fill (\i,0.5) circle (1.2pt); 
    }


    \end{tikzpicture}~~~~~~~~~~~ 
    \begin{tikzpicture}[>=stealth,scale=0.5]
    \foreach \i in {1,...,4} {
        \node at (\i,0) {\i};
        \fill (\i,0.5) circle (1.2pt); 
    }

    \draw[-,bend left=40] (1,0.5) to (2,0.5);

    \end{tikzpicture}~~~~~~~~~~~ 
    \begin{tikzpicture}[>=stealth,scale=0.5]
    \foreach \i in {1,...,4} {
        \node at (\i,0) {\i};
        \fill (\i,0.5) circle (1.2pt); 
    }

    \draw[-,bend left=40] (1,0.5) to (3,0.5);

    \end{tikzpicture}~~~~~~~~~~~ 
        \begin{tikzpicture}[>=stealth,scale=0.5]
    \foreach \i in {1,...,4} {
        \node at (\i,0) {\i};
        \fill (\i,0.5) circle (1.2pt); 
    }

    \draw[-,bend left=40] (1,0.5) to (4,0.5);

    \end{tikzpicture}~~~~~~~~~~~ 
    \begin{tikzpicture}[>=stealth,scale=0.5]
    \foreach \i in {1,...,4} {
        \node at (\i,0) {\i};
        \fill (\i,0.5) circle (1.2pt); 
    }

    \draw[-,bend left=40] (2,0.5) to (3,0.5);

    \end{tikzpicture}~~~~~~~~~~~ 
    
\scalebox{0.75}{
\begin{ytableau}
1 & 2 & 3 &4 
\end{ytableau}~~~~~~~~~~~~~~
\begin{ytableau}
1 & 3 & 4 \\
2
\end{ytableau}~~~~~~~~~~~~~~~~~~
\begin{ytableau}
1 & 4\\
2\\
3
\end{ytableau}~~~~~~~~~~~~~~~~~~~~~
\begin{ytableau}
1 & 3 \\
2\\
4
\end{ytableau}~~~~~~~~~~~~~~~~~~
\begin{ytableau}
1 & 2 & 4 \\
3  \\
\end{ytableau}
}

\vspace{20pt}
    \begin{tikzpicture}[>=stealth,scale=0.5]
    \foreach \i in {1,...,4} {
        \node at (\i,0) {\i};
        \fill (\i,0.5) circle (1.2pt); 
    }

    \draw[-,bend left=40] (2,0.5) to (4,0.5);

    \end{tikzpicture}~~~~~~~~~~~ 
        \begin{tikzpicture}[>=stealth,scale=0.5]
    \foreach \i in {1,...,4} {
        \node at (\i,0) {\i};
        \fill (\i,0.5) circle (1.2pt); 
    }

    \draw[-,bend left=40] (3,0.5) to (4,0.5);

    \end{tikzpicture}~~~~~~~~~~ 
        \begin{tikzpicture}[>=stealth,scale=0.5]
    \foreach \i in {1,...,4} {
        \node at (\i,0) {\i};
        \fill (\i,0.5) circle (1.2pt); 
    }

    \draw[-,bend left=40] (1,0.5) to (2,0.5);
    \draw[-,bend left=40] (3,0.5) to (4,0.5);

    \end{tikzpicture}~~~~~~~~~~~ 
            \begin{tikzpicture}[>=stealth,scale=0.5]
    \foreach \i in {1,...,4} {
        \node at (\i,0) {\i};
        \fill (\i,0.5) circle (1.2pt); 
    }

    \draw[-,bend left=40] (1,0.5) to (4,0.5);
    \draw[-,bend left=40] (2,0.5) to (3,0.5);

    \end{tikzpicture}~~~~~~~~~
                \begin{tikzpicture}[>=stealth,scale=0.5]
    \foreach \i in {1,...,4} {
        \node at (\i,0) {\i};
        \fill (\i,0.5) circle (1.2pt); 
    }

    \draw[-,bend left=40] (1,0.5) to (3,0.5);
    \draw[-,bend left=40] (2,0.5) to (4,0.5);

    \end{tikzpicture}~~~~~~~~~
    
    \scalebox{0.75}{
    ~~~~~\begin{ytableau}
    1 & 2 \\ 3 \\ 4 
    \end{ytableau}~~~~~~~~~~~~~~~~~~~
    \begin{ytableau}
    1 & 2 & 3 \\
    4
    \end{ytableau}~~~~~~~~~~~~~~~~~~
    \begin{ytableau}
    1 & 3\\
    2 & 4\\
    \end{ytableau}~~~~~~~~~~~~~~~~~~~~~~~~
    \begin{ytableau}
    1 \\
    2\\
    3\\
    4
    \end{ytableau}~~~~~~~~~~~~~~~~~~~~~~
        \begin{ytableau}
    1 & 2\\
    3 & 4
    \end{ytableau}
    }

\caption{Insertion tableaux of partial matchings on $[4]$. }\label{fig-xc}
\end{figure}
There are a total of ten partial matchings. The rightmost one in the second row is not noncrossing, whose insertion tableau is the only standard tableau of size $4$ which is not a Richarson tableau. Moreover, among the nine Richardson tableaux, two are prime tableaux which exactly correspond  to the two prime noncrossing partial matchings drawn   in the fourth column. 

\begin{proof}[Proof of Theorem \ref{hg-x}]
 The proof is   by induction on size $n$. For $n=1$ or $2$, this is clear as shown in Figure \ref{cxh-09}. 
\begin{figure}[h t]
    \begin{center}
    ~    \begin{tikzpicture}[>=stealth,scale=0.5]
    \foreach \i in {1} {
        \node at (\i,0) {\i};
        \fill (\i,0.5) circle (1.2pt); 
    }


    \end{tikzpicture}~~~~~~~~~~
       \begin{tikzpicture}[>=stealth,scale=0.5]
    \foreach \i in {1,...,2} {
        \node at (\i,0) {\i};
        \fill (\i,0.5) circle (1.2pt); 
    }


    \end{tikzpicture}~~~~~~~~~~~~~
       \begin{tikzpicture}[>=stealth,scale=0.5]
    \foreach \i in {1,...,2} {
        \node at (\i,0) {\i};
        \fill (\i,0.5) circle (1.2pt); 
    }

    \draw[-,bend left=40] (1,0.5) to (2,0.5);

    \end{tikzpicture}
    \end{center}

   ~~~~~~~~~~~~~~~~~~~~~~~~~~~~~~~~~~~ \begin{ytableau}
        1
    \end{ytableau}~~~~~~~~~
    \begin{ytableau}
        1 & 2
    \end{ytableau}~~~~~~~~~~~~~
    \begin{ytableau}
        1 \\ 2
    \end{ytableau}
  \caption{The correspondence $\Phi$ for $n=1,2$.}  \label{cxh-09}
\end{figure}
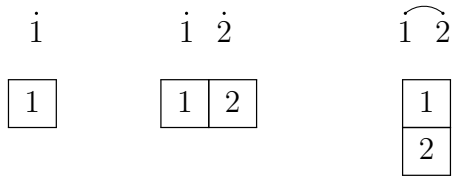
Assume now that $n>2$.
We first show that $\Phi$ is an injection from $\mathrm{NC}(n)$ to $\mathrm{RT}(n)$. 
Let $w$ be a noncrossing involution of $[n]$, and $T$ be its insertion tableau. We need to verify that  $T$ is a Richardson tableau.  There are two cases.

\begin{itemize}
    \item[(1)] The noncrossing involution  $w$ is not   prime.
    In this case, we decompose  $w=w^1\oplus w^2\oplus \cdots \oplus w^k$   into prime noncrossing involutions.     
    It follows from Lemma \ref{lem-23} that $T=T_1\circ T_2\circ \cdots \circ T_k$, where $T_i$ is the insertion tableau of the prime piece $w^i$. By induction, each $T_i$ is a Richardson tableau. By \cite[Lemma 3.9]{KP}, the concatenation of Richardson tableaux is also a Richardson tableau. This  concludes  that $T$ is a Richardson tableau.

    \item[(2)] The noncrossing involution  $w$ is    prime. In this case, we have $w(1)=n$ and $w(n)=1$.  Let $U$ be the insertion tableau  of  $w(1)w(2)\cdots w(n-1)$, and $U'$ be the insertion tableau of $w(2)w(3)\cdots w(n-1)$. Since $w(1)=n$, it is easily seen that $U$ is simply obtained from $U'$ by placing a box filled with $n$ as a new row on the bottom.  Now $T$ is obtained  by inserting $w(n)=1$ into $U$. By the construction of the RS algorithm,  $T$ is obtained from $U$ by moving each entry in the first column down by one unit, and then placing $1$ into the top-left corner box. 

    Note that the tableau $U'$ is the insertion tableau of a noncrossing involution of $\{2,3,\ldots, n-1\}$. By induction, $U'$ satisfies the Richardson condition. Since $U$ is obtained by placing $n$ on the bottom of $U'$, we see that $U$ also satisfies the Richardson condition. Furthermore, since $T$ is obtained from $U$ by shifting every entry in the first column down, it is easily checked  that $T$ also satisfies the Richardson condition. So $T$ is a Richardson tableau. 

Let us give an example to explain the above construction for the  prime case.  Figure \ref{fig-pq}, lists a prime noncrossing partial matching on $[8]$, whose associated involution is $84325761$, together with its  corresponding tableaux $T$, $U$ and $U'$. 
\begin{figure}[h t]
  \begin{center}
    \begin{tikzpicture}[>=stealth,scale=1]
    \foreach \i in {1,...,8} {
        \node at (\i,0) {\i};
        \fill (\i,0.4) circle (1.2pt); 
    }

    \draw[-,bend left=40] (1,0.4) to (8,0.4);
    \draw[-,bend left=30] (2,0.4) to (5,0.4);
    \draw[-,bend left=30] (6,0.4) to (7,0.4);

\end{tikzpicture}
\end{center}

\begin{center}
    \begin{ytableau}
1 & 4 & 6 \\
2 & 7 \\
3 \\
5\\
8\\
\none[]\\
\none[]& \none[T=U\leftarrow 1]
\end{ytableau}
\qquad  \qquad
\begin{ytableau}
2 & 4 & 6 \\
3 & 7 \\
5 \\
8\\
\none[]\\
\none[]\\
\none[]& \none[U]
\end{ytableau}
\qquad \qquad
\begin{ytableau}
2 & 4 & 6 \\
3 & 7 \\
5\\
\none[]\\
\none[]\\
\none[]\\
\none[]& \none[U']
\end{ytableau}
\end{center}
\caption{An illustration of the construction for the prime case.}\label{fig-pq}
\end{figure}
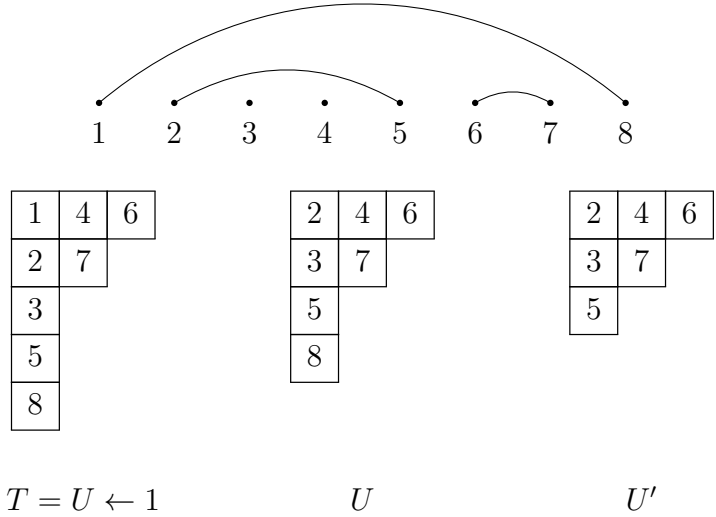

\end{itemize}

To complete the proof, we need to verify that for each Richardson tableau $T\in \mathrm{RT}(n)$, there exists a noncrossing involution $w$ of $[n]$ whose insertion tableau equals $T$. The proof is also by induction on $n$. Similarly, the arguments are  divided into two cases.
\begin{itemize}
    \item[(1)] The Richardson tableau $T$ is not  prime.       Write  $T=T_1\circ T_2\circ \cdots \circ T_k$ for the decomposition into prime Richardson tableaux. By induction, there exists a unique noncrossing involution  $w^i$ whose insertion tableau is $T_i$. Set  $w=w^1\oplus w^2\oplus \cdots \oplus w^k$. By  Lemma \ref{lem-23}, we see that  $\Phi(w)=T$.

    \item[(2)] The Richardson tableau $T$ is  prime. By \cite[Corollary 3.15]{KP}, the shape of $T$ is of   form    $\lambda=(\lambda_1,\ldots, \lambda_{\ell-2},1,1)$, and moreover it follows from  \cite[Corollary 3.16]{KP} that the largest number $n$ appears in the bottom row. 
Let $T(i,j)$ denote the entry of $T$ lying in row $i$ and column $j$.  Let $U$ be the standard  tableau on $\{1,2\ldots, n-1\}$ obtained from $T$ by removing the lowest  box containing $n$. Then $U$ has $\ell-1$ rows, and the last row of $U$ has one box.  The following claim is a key ingredient. 

{\textbf{Claim.} } 
Fix any row index  $2\leq i\leq \ell-1$. Suppose that the box $(i-1,2)$ also lies in $U$. Then we have $U(i,1)<U(i-1, 2)$. 

The above claim can be proved by induction on $n$.    
 By the  definition of a Richardson tableau, it is clear  that  $U$ is also a Richardson tableau.  Making use of 
\cite[Proposition 3.13]{KP}, there exists a unique  decomposition $U=P\circ Q$ where $P$ is a prime Richardson tableau with exactly  $\ell-1$ rows, and $Q$ is a Richardson tableau with at most $\ell-2$ rows. So the box $(i,1)$ is in $P$. If  $(i-1,2)$ lies in $Q$, the claim is clearly true since every entry in $Q$ is greater than all the entries in $P$. It remains to consider the case when $(i-1,2)$ lies in $P$. In this case, since $P$ is prime,  the shape of $P$ has the form $\mu=(\mu_1,\ldots, \mu_{\ell-3}, 1,1)$. This, along with the assumption that $(i-1,2)$ lies in $P$, implies that  $i<\ell-1$. Therefore, we may apply induction to $P$ and conclude that $P(i,1)<P(i-1,2)$.  This verifies the claim since $U(i,1)=P(i,1)$ and $U(i-1,2)=P(i-1,2)$. 

Now, we locate the lowest box $(\ell-1,1)$ of $U$, and complement the reverse procedure of the RSK algorithm. According to the above claim, we see that the `bumping route' is exactly in the first column. This implies $U=U'\leftarrow 1$ where $U'$ is the tableau obtained from $U$ by deleting the entry $1$ in   the first column, and then moving the remaining entries in the first column up by one unit. 

Since $U$ is a Richardson tableau, by the above construction, it is easy to check that $U'$ is tableau on $\{2,\ldots, n-1\}$ which satisfies the Richardson condition. Hence, by induction, there exists a unique noncrossing involution $u$ of $\{2,\ldots, n-1\}$ such that the insertion tableau of $u$ is $U'$. We set $w=n\cdot u\cdot  1$ to be the prime noncrossing involution of $[n]$ by appending $n$ and $1$ on the left and right sides of $u$.  Noticing that the insertion tableau of $w$ is generated from $U$ by placing $n$ as a single row on the bottom, we conclude that the insertion  tableau of $w$ is $T$.     
\end{itemize}

Combining the above, we see that $\Phi$ is indeed a bijection,  completing  the proof.     
\end{proof}

By the construction in the proof of Theorem \ref{hg-x}, we have the following observation. 

\begin{coro}\label{gh-12}
    If restricting $\Phi$ to the set of prime noncrossing involutions of $[n]$, then the image set is exactly the set of prime Richardson tableaux of size $n$. 
\end{coro}

\begin{re}\label{re-xj}
Notice that there is a bijection between the set of prime noncrossing partial matchings on $[n]$ and the set of noncrossing matchings on $[n-2]$ by removing the arc $(1,n)$. So Corollary \ref{gh-12} induces a bijection, denoted  $\Psi$, between   the set of prime Richardson tableaux of size $n$ and the set of Richardson tableaux of size $n-2$. Moreover, by the construction in the proof of Theorem \ref{hg-x}, $\Psi$ will send a prime Richardson tableau of shape $\lambda=(\lambda_1,\ldots, \lambda_{\ell-2},1,1)$  to a Richardson tableau  of shape $\lambda=(\lambda_1,\ldots, \lambda_{\ell-2})$. So $\Psi$ yields an alternative proof of \cite[Corollary 4.3]{KP} which is restated in Corollary  \ref{coro-22}.   
\end{re}

\footnotesize{

\textsc{(Peter L. Guo) Center for Combinatorics, Nankai University, LPMC, Tianjin 300071, P.R. China}

{\it
Email address: \tt lguo@nankai.edu.cn}

\end{document}